\documentclass[]{article}

\usepackage[utf8]{inputenc}
\usepackage[margin=1.25in]{geometry}

\usepackage{amsthm, amsfonts, amsmath, makecell, tikz, 
	graphicx, colortbl}
\usepackage{multirow, booktabs, bigstrut}

\usepackage{mathtools}

\usepackage{graphicx}
\usepackage{float}
\usepackage{color}
\usepackage{tikz,pgfplots}
\usepackage{tikz-cd}
\usepackage{verbatim}
\usepackage{enumerate}
\usepackage{physics}
\usepackage{braket}

\usepackage{bm}

\usepackage{enumerate, amssymb}

\usetikzlibrary{graphs,graphs.standard,arrows.meta,calc,intersections}

\usepackage[]{algorithm, algorithmic}

\usepackage{tabularx}

\theoremstyle{plain}
\newtheorem{theorem}{Theorem}[section]
\newtheorem{lemma}[theorem]{Lemma}
\newtheorem{corollary}[theorem]{Corollary}
\newtheorem{proposition}[theorem]{Proposition}

\newtheorem{remark}[theorem]{Remark}

\theoremstyle{definition}
\newtheorem{definition}[theorem]{Definition}

\DeclareMathOperator{\Int}{Int}

\DeclareMathOperator{\IntR}{Int{}^\text{R}}

\DeclareMathOperator{\minval}{minval}
\DeclareMathOperator{\loc}{loc}

\DeclareMathOperator{\Spec}{Spec}

\renewcommand{\epsilon}{\varepsilon}

\newcommand{\N}{{\mathbb N}}

\newcommand{\Q}{{\mathbb Q}}

\newcommand{\Z}{{\mathbb Z}}

\renewcommand{\P}{\mathfrak{P}}
\newcommand{\p}{\mathfrak{p}}
\newcommand{\q}{\mathfrak{q}}
\newcommand{\m}{\mathfrak{m}}
\newcommand{\M}{\mathfrak{M}}

\definecolor{gray}{rgb}{.5,.5,.5}

\definecolor{black}{rgb}{0,0,0}

\definecolor{blue}{rgb}{0,0,1}

\definecolor{red}{rgb}{1,0,0}

\definecolor{green}{rgb}{0,1,0}

\definecolor{gold}{rgb}{.5,.5,.2}

\definecolor{yellow}{rgb}{1,1,.4}

\definecolor{purple}{rgb}{.5,0,.5}

\definecolor{darkgreen}{rgb}{0,.5,0}

\definecolor{orange}{rgb}{1,.55,0}

\definecolor{white}{rgb}{1,1,1}

\usepackage[colorinlistoftodos]{todonotes}

\let\originalleft\left
\let\originalright\right
\renewcommand{\left}{\mathopen{}\mathclose\bgroup\originalleft}
\renewcommand{\right}{\aftergroup\egroup\originalright}

\let\originaltodo\todo
\renewcommand{\todo}[1]{\originaltodo[inline]{#1}}

\usepackage{hyperref}
\usepackage{xcolor}
\hypersetup{
	colorlinks,
	linkcolor={red!50!black},
	citecolor={blue!50!black},
	urlcolor={blue!80!black}
}

\usepackage{subfigure}

\title{Integer-valued rational functions over globalized pseudovaluation domains}

\author{Baian Liu}

\begin{document}
	
	\maketitle

\begin{abstract}
	Let $D$ be a domain. Park determined the necessary and sufficient conditions for which the ring of integer-valued polynomials $\Int(D)$ is a globalized pseudovaluation domain (GPVD). In this work, we investigate the ring of integer-valued rational functions $\IntR(D)$. Since it is necessary that $D$ be a GPVD for $\IntR(D)$ to be a GPVD, we consider $\IntR(D)$, where $D$ is a GPVD. We determine that if $D$ is a pseudosingular GPVD, then $\IntR(D)$ is a GPVD. We also completely characterize when $\IntR(D)$ is a GPVD if $D$ is a pseudovaluation domain that is not a valuation domain. 
\end{abstract}

\section{Introduction}
\indent\indent The concept of integer-valued polynomials has been studied throughout many different areas of mathematics. One way to study integer-valued polynomials is to consider a collection of integer-valued polynomials as a ring. Given a domain $D$ with field of fractions $K$ and $E$ some subset of $K$, we can define \[
\Int(D) \coloneqq \{ f \in K[x]\mid f(d) \in D,\, \forall d \in D \} \quad \text{and} \quad	\Int(E,D) \coloneqq \{ f \in K[x] \mid f(a) \in D,\,\forall a \in E \}
\]
the \textbf{ring of integer-valued polynomials over $D$} and \textbf{ring of integer-valued polynomials on $E$ over $D$}, respectively. Note that $\Int(D, D) = \Int(D)$.

A type of question one can ask about the rings of the form $\Int(D)$ starts by fixing a property on the rings of integer-valued polynomials. Then we investigate what kind of conditions the base ring $D$ must have. The two papers \cite{IntPruferClassification} and \cite{CahenChabertFrisch} give two different classifications of Prüfer domains of the form $\Int(D)$. Furthermore, \cite{PvMD} gives a complete characterization of when $\Int(D)$ is a Prüfer $v$-multiplication domain (P$v$MD). The ring-theoretic property we look at in this work is the property of being a globalized pseudovaluation domain (GPVD). In \cite{ParkGPVD}, Park gives complete necessary and sufficient conditions for $\Int(D)$ to be a GPVD. To introduce this result, we provide the definitions surrounding GPVDs.

We can consider a GPVD to be a generalization of a Prüfer domain. We first focus on the local counterparts. Localizing a Prüfer domain at a prime ideal yields a valuation domain. A way to generalize a valuation domain is to use a pseudovaluation domain. For references on pseudovaluation domains, see \cite{HedstromHouston, HedstromHoustonII}. 
\begin{definition}
	A domain $D$ is a \textbf{pseudovaluation domain} (PVD) if $D$ has a valuation overring $V$ such that $\Spec(D) = \Spec(V)$ as sets. The valuation domain is uniquely determined and is called the \textbf{associated valuation domain} of $D$.
\end{definition}

\begin{remark}
	In particular, a pseudovaluation domain and the associated valuation domain have the same maximal ideal. 
	
	One way to construct a pseudovaluation domain is to start with a valuation domain $V$. Let $\m$ be the maximal ideal of $V$. Consider the canonical projection $\pi:V \to V/\m$. Take a subfield $F \subseteq V/\m$. Then $D \coloneqq \pi^{-1}(F)$ is a pseudovaluation domain with associated valuation domain $V$. 
\end{remark}

We can consider Prüfer domains to be global counterparts of valuation domains. Dobbs and Fontana studied two global counterparts of pseudovaluation domains called locally pseudovaluation domains and globalized pseudovaluation domains \cite{DobbsFontana}. First, we introduce the definition of a locally pseudovaluation domain.

\begin{definition}
	A domain $D$ is a \textbf{locally pseudovaluation domain} (LPVD) if for every maximal ideal $\m$ of $D$, the localization $D_\m$ is a PVD.
\end{definition}

Since there is a valuation domain associated with a pseudovaluation domain, we would like there to be a Prüfer domain associated with a locally pseudovaluation domain. However, a locally pseudovaluation domain that is not a PVD or a Prüfer domain does not have a Prüfer overring with the same prime spectrum \cite[Proposition 3.3]{AndersonDobbs}. Nevertheless, there is a subclass of LPVDs that does have an associated Prüfer overring which is a unibranched extension. 

\begin{definition}
	An extension of commutative rings $A \subseteq B$ is \textbf{unibranched} if the contraction map $\Spec(B) \to \Spec(A)$ is a bijection.
	
	A domain $D$ is a \textbf{globalized pseudovaluation domain} (GPVD) if there exists a Prüfer domain $T$ containing $D$ such that 
	\begin{itemize}
		\item 	$D \subseteq T$ is a unibranched extension and
		\item 	 there exists a nonzero radical ideal $J$ common to $D$ and $T$ such that each prime ideal of $T$ containing $J$ is maximal in $T$ and each prime ideal of $D$ containing $J$ is maximal in $D$.
	\end{itemize}
	The Prüfer domain is uniquely determined and is called the \textbf{associated Prüfer domain} of $D$. 
\end{definition}

\begin{remark}
	Theorem 3.1 in \cite{DobbsFontana} provides an equivalent definition of a GPVD. A domain $D$ is a GPVD if there exists a Prüfer domain $T$ containing $D$ such that there is a common radical ideal $J$ where $D/J \subseteq T/J$ is a unibranched extension of Krull dimension zero rings. 
\end{remark}

\begin{remark}
	A GPVD is an LPVD, and a PVD is a GPVD with its associated valuation domain as its associated Prüfer domain and their common maximal ideal as the common radical ideal in the definition of a GPVD \cite{DobbsFontana}. 
\end{remark}

We return to Park's characterization of GPVDs of the form $\Int(D)$. This characterization uses the idea of an interpolation domain, which is a domain $D$ such that for every distinct pair $a, b \in D$ there exists $f \in \Int(D)$ such that $f(a) = 0$ and $f(b) = 1$. Park showed that for a domain $D$ that is not a field, $\Int(D)$ is a GPVD if and only if $D$ is GPVD and an interpolation domain \cite{ParkGPVD}. Furthermore, when $\Int(D)$ is a GPVD, the associated Prüfer domain is $\Int(D, T)$, where $T$ is the associated Prüfer domain of $D$. 

Now we examine a generalization of the concept of integer-valued polynomials. We will study integer-valued rational functions. For a domain $D$ with field of fractions $K$ and $E$ some subset of $K$, we define 	\[
\IntR(D) \coloneqq \{ \varphi \in K(x) \mid \varphi(d) \in D,\, \forall d \in D \} \quad \text{and} \quad	\IntR(E,D) \coloneqq \{ \varphi \in K(x) \mid \varphi(a) \in D,\,\forall a \in E \}
\]
the \textbf{ring of integer-valued rational functions over $D$} and the \textbf{ring of integer-valued rational functions on $E$ over $D$}, respectively. Note that $\IntR(D,D) = \IntR(D)$. We can also define an ideal of $\IntR(E,D)$ using an ideal $I$ of $D$. One can check that the set 
\[
	\IntR(E,I) \coloneqq \{\varphi \in \IntR(E,D) \mid \varphi(a) \in I,\,\forall a \in E \}
\]
is an ideal of $\IntR(E,D)$. 

We want to explore under what conditions the ring of integer-valued rational functions is a GPVD. We first provide a necessary condition. As with Prüfer domains, the homomorphic image of a GPVD is a GPVD \cite[Lemma 1]{ParkGPVD}. In order to have a ring of integer-valued rational functions that is a GPVD, we must have a base ring that is a GPVD.

\begin{proposition}
	Let $D$ be a domain with $E$ a nonempty subset of the field of fractions of $D$. If $\IntR(E,D)$ is a GPVD, then $D$ is a GPVD. 
\end{proposition}

\begin{proof}
	Let $a \in E$ be any element and consider the homomorphism $\IntR(E,D) \to D$ given by evaluation at $a$. Then $D$ is the homomorphic image of $\IntR(E,D)$ so $D$ is a GPVD. 
\end{proof}

In Section \ref{Sect:Pseudosingular}, we introduce the notion of a pseudosingular GPVD, which generalizes the idea of a singular Prüfer domain. We then show that $\IntR(E,D)$ is a GPVD for $D$ a pseudosingular GPVD and $E$ any subset of the field of fractions of $D$. We also show that the Prüfer domain associated with $\IntR(E,D)$ is $\IntR(E,T)$, where $T$ is the Prüfer domain associated with $D$. In Section \ref{Sect:Non-singular}, we completely classify which PVDs $D$, which are not valuation domains, make the ring $\IntR(D)$ a GPVD. In Section \ref{Sect:Local}, we give necessary and sufficient conditions under which the ring $\IntR(K,D)$ is a local domain, where $D$ is a PVD that is not a valuation domain and $K$ is the field of fractions of $D$. Lastly, Section \ref{Sect:FieldMaps} is dedicated to proving a lemma about rational function mappings between fields that is used in Section \ref{Sect:Local}.

\section{Pseudosingular GPVDs}\label{Sect:Pseudosingular}

\indent\indent In this section, we give a family of GPVDs such that their rings of integer-valued rational functions is also a GPVD. We introduce the notion of a pseudosingular GPVD, generalizing notion of a singular Prüfer domain. The notion of a singular Prüfer domain is important as it gives a condition under which the ring $\IntR(E,D)$ is a Prüfer domain. The same proof for Theorem 3.5 in \cite{IntValuedRational} shows that if $D$ is a singular Prüfer domain, then $\IntR(E,D)$ is a Prüfer domain for any subset $E$ of the field of fractions of $D$. For this reason, we generalize this notion to give a condition under which the ring $\IntR(E,D)$ is a GPVD.

\begin{definition}\cite{IntValuedRational}
	Let $D$ be a Prüfer domain. We say $D$ is \textbf{singular} if there exists a family $\Lambda$ of maximal ideals of $D$ so that
	\begin{itemize}
		\item $D = \bigcap\limits_{\m \in \Lambda} D_\m$;
		\item for each $\m \in \Lambda$, the maximal ideal of $D_\m$ is principal, generated by some $t_\m \in D_\m$; and
		\item there exists some $t \in D$ and $n \in \N$ such that for each $\m \in \Lambda$, we have $0 < v_\m(t) < nv_\m(t_\m)$, where $v_\m$ is a valuation associated with $D_\m$. 
	\end{itemize}
\end{definition}

Since there is a Prüfer domain associated with a GPVD, we can impose conditions on the associated Prüfer domain as conditions on the GPVD. 

\begin{definition}
	Let $D$ be a GPVD. We say that $D$ is \textbf{pseudosingular} if the associated Prüfer domain $T$ is singular.
\end{definition}

Because a GPVD has an associated Prüfer domain that is a unibranched extension, it could be useful to have a description of the maximal ideals of the ring we suspect is a GPVD in order to prove that the ring is indeed a GPVD. Some of the maximal ideals of $\IntR(E,D)$ can be defined through maximal ideals of the base ring $D$. 

\begin{definition}
	Let $D$ be a domain and let $E$ be some nonempty subset of the field of fractions of $D$. Take $\m$ to be a maximal ideal of $D$ and $a \in E$. We define the set
	\[
		\M_{ \m, a } \coloneqq \{\varphi \in \IntR(E,D) \mid \varphi(a) \in \m \}.
	\]
	We know that $\M_{ \m, a}$ is a maximal ideal of $\IntR(E,D)$ because it is the kernel of the surjective map $\IntR(E,D) \to D/\m$ given by evaluation at $a$ modulo $\m$. Maximal ideals of $\IntR(E,D)$ of the form $\M_{\m, a}$ are called \textbf{maximal pointed ideals}.
\end{definition}

In general, not all of the maximal ideals of $\IntR(E,D)$ are maximal pointed ideals. Since there is a notion of a limit of a family of ideals using filters and ultrafilters, we can use this notion to potentially describe more maximal ideals of $\IntR(E,D)$.

\begin{definition}
	Let $S$ be a set. A \textbf{filter} $\mathcal{F}$ on $S$ is a collection of subsets of $S$ such that
	\begin{enumerate}
		\item $\emptyset \notin \mathcal{F}$;
		\item if $A, B \in \mathcal{F}$, then $A \cap B \in \mathcal{F}$; and
		\item if $A \in \mathcal{F}$ and $B \subseteq S$ is such that $A \subseteq B$, then $B \in \mathcal{F}$. 
	\end{enumerate}
	If $\mathcal{U}$ is a filter on $S$ such that for every $A \subseteq S$, we have $A \in \mathcal{U}$ or $S \setminus A \in \mathcal{U}$, then we call $\mathcal{U}$ an \textbf{ultrafilter}. Every filter of $S$ is contained in some ultrafilter of $S$ by the Ultrafilter Lemma. 
\end{definition}


\begin{definition}
	Let $R$ be a commutative ring. Take $\{I_\lambda\}_{\lambda \in \Lambda}$ to be a family of ideals of $R$. For each $r \in R$, we define the \textbf{characteristic set of $r$ on $\{I_\lambda\}$} to be
	\[
	\chi_r = \{I_\lambda \mid r \in I_\lambda \}.
	\]
	For a filter $\mathcal{F}$ on $\{I_\lambda\}$, we define the \textbf{filter limit of $\{I_\lambda\}$ with respect to $\mathcal{F}$} as
	\[
	\lim\limits_{\mathcal{F}} I_\lambda = \{r \in R \mid \chi_r \in \mathcal{F} \}.
	\]
	If $\mathcal{F}$ is an ultrafilter, we call $\lim\limits_{\mathcal{F}} I_\lambda$ the \textbf{ultrafilter limit of $\{I_\lambda\}$ with respect to $\mathcal{F}$}.
\end{definition}

\begin{remark}
	The filter and ultrafilter limits of a family of ideals are also themselves ideals. If $\{\p_\lambda\}_{\lambda \in \Lambda}$ is a family of prime ideals and $\mathcal{U}$ is an ultrafilter of $\{\p_\lambda\}$, then the ultrafilter limit $\lim\limits_{\mathcal{U}} \p_\lambda$ is also a prime ideal. This gives rise to the \textbf{ultrafilter topology} on $\Spec(R)$, which is identical to the patch topology and the constructible topology on $\Spec(R)$ \cite{UltrafilterTopology}. 
\end{remark}

We will now use a description of the maximal ideals using maximal pointed ideals and ultrafilter limits to show that rings of integer-valued rational functions over pseudosingular GPVDs are also GPVDs. We will also utilize the rational function $\theta(x) =\frac{t(1+x^{2n})}{(1+tx^n)(t+x^n)}$, where $t$ and $n$ come from the definition of a singular Prüfer domain. This rational function played an integral role in showing that a ring of integer-valued rational functions over a singular Prüfer domain is a Prüfer domain in \cite{IntValuedRational}. 

\begin{theorem}\label{Thm:Pseudosingular}
	Let $D$ be a pseudosingular GPVD. Denote by $K$ the field of fractions of $D$ and let $E$ be a nonempty subset of $K$. Suppose that $T$ is the associated Prüfer domain of $D$. Then $\IntR(E,D)$ is a GPVD with associated Prüfer domain $\IntR(E,T)$. 
%
\end{theorem}

\begin{proof}
	Let $J$ denote the nonzero radical ideal common to $D$ and $T$ so that $D/J \subseteq T/J$ is a unibranched extension of Krull dimension zero rings. Such an ideal exists since $D$ is a GPVD.

	Since $D$ is pseudosingular, we know that $T$ is singular. Therefore, there exists a collection $\Lambda$ of maximal ideals of $T$ such that
	\begin{itemize}
		\item $T = \bigcap\limits_{\m \in \Lambda} T_\m$,
		\item for each $\m \in \Lambda$, the maximal ideal of the valuation domain $T_\m$ is principally generated by some $t_\m \in T_\m$, and
		\item there exists some $t \in T$ and $n \in \N$ such that $0 < v_\m(t) < nv_\m(t_\m) $ for all $\m \in \Lambda$, where $v_\m$ is a valuation associated with $T_\m$. 
	\end{itemize}
	We fix $t$ and $n$ for the rest of proof, as well as $v_\m$ for each $\m \in \Lambda$. Furthermore, since we are considering two rings of integer-valued rational functions at once, we need notation to clarify of which ring the maximal pointed ideals are ideals. Let $a \in E$ and $\m \in \Lambda$. We define $\M_{\m, a}^T \coloneqq \{\varphi \in \IntR(E,T) \mid \varphi(a) \in \m \}$ as an ideal of $\IntR(E,T)$. Let $\mathfrak{n} \coloneqq \m \cap D$. We define $\M_{\mathfrak{n}, a}^D \coloneqq \{\varphi \in \IntR(E, D) \mid \varphi(a) \in \mathfrak{n} \}$ as an ideal of $\IntR(E,D)$. 
	
	We first give the statements we want to prove. The proofs of the statements will follow. 
	
	\begin{enumerate}
		\item 	Every maximal ideal of $T$ is an ultrafilter limit of ideals in $\Lambda$.  Furthermore, $t$ is in the Jacobson radical of $T$. 
		
		\item 	The ideal $\IntR(E,J)$ is a nonzero radical ideal of both $\IntR(E,D)$ and $\IntR(E,T)$. 
		
		\item 	A prime ideal of $\IntR(E,T)$ containing $\IntR(E,J)$ is maximal in $\IntR(E,T)$. 
		
		\item 	A prime ideal of $\IntR(E,D)$ containing $\IntR(E,J)$ is maximal in $\IntR(E,D)$. 
		
		\item 	Every maximal ideal of $\IntR(E,T)$ can be written in the form $\lim\limits_{\mathcal{U}} \M_{\m,a}^T$ for some ultrafilter $\mathcal{U}$ of $\{\M_{\m, a}^T \mid a \in E, \m \in \Lambda\}$.

		\item Every maximal ideal of $\IntR(E,D)$ can be written in the form $\lim\limits_{\mathcal{U}} \M_{\mathfrak{n},a}^D$ for some ultrafilter $\mathcal{U}$ of $\{\M_{\mathfrak{n}, a}^D \mid a \in E, \m \in \Lambda, \mathfrak{n} = \m \cap D \}$.
		
		\item Suppose that $\mathcal{U}_1$ and $\mathcal{U}_2$ are ultrafilters of $\{\M_{\m, a}^T \mid a \in E, \m \in \Lambda\}$ such that $\lim\limits_{\mathcal{U}_1} \M_{\m,a}^T$ and $\lim\limits_{\mathcal{U}_2} \M_{\m,a}^T$ are distinct maximal ideals of $\IntR(E,T)$. Then $\lim\limits_{\mathcal{U}_1} \M_{\m,a}^T \cap \IntR(E, D) \neq \lim\limits_{\mathcal{U}_2} \M_{\m,a}^T \cap \IntR(E,D)$.

		\item 	The ring $\IntR(E,D)$ is a GPVD with associated Prüfer domain $\IntR(E,T)$.
	\end{enumerate}
	
	The first four items are to establish that there is a nonzero radical ideal, namely $\IntR(E,J)$, common to $\IntR(E,D)$ and $\IntR(E,T)$ such that each prime ideal of $\IntR(E,D)$ or $\IntR(E,T)$ containing that common nonzero radical ideal is actually a maximal ideal of the ring to which it belongs. The next three statements shows that $\IntR(E,D)/\IntR(E,J) \subseteq \IntR(E,T)/\IntR(E,J)$ is a unibranched extension. Now we prove the claims. 
	
	\begin{enumerate}
		\item Claim: Every maximal ideal of $T$ is an ultrafilter limit of ideals in $\Lambda$.  Furthermore, $t$ is in the Jacobson radical of $T$. 
		
		Since $T$ is a Prüfer domain, every nonzero ideal of $T$ is a $t$-ideal. In particular, every maximal ideal of $T$ is a $t$-ideal. Writing $T$ as  $T = \bigcap\limits_{\m \in \Lambda} T_\m$, we see that every maximal ideal of $T$ is an ultrafilter limit of ideals in $\Lambda$ \cite[Proposition 2.8]{PvMD}.
		
		Now let $\mathfrak{a}$ be a maximal ideal of $T$. We know that $\mathfrak{a}$ is an ultrafilter limit of ideals in $\Lambda$. Since $t \in \m$ for all $\m \in \Lambda$, we also have that $t \in \mathfrak{a}$. Since $t$ is in all of the maximal ideas of $T$, we can conclude that $t$ is in the Jacobson radical of $T$.
		
		\item 	Claim: The ideal $\IntR(E,J)$ is a nonzero radical ideal of both $\IntR(E,D)$ and $\IntR(E,T)$. 
		
		Since $J$ is a nonzero ideal of both $D$ and $T$, we know that $\IntR(E,J)$ is a nonzero ideal of both $\IntR(E,D)$ and $\IntR(E,T)$. Now suppose that $\varphi \in \IntR(E,D)$ is such that there exists an $m \in \N$ such that $\varphi^m \in \IntR(E,J)$. Then for any $a \in E$, we have $\varphi(a)^m \in J$. Since $\varphi(a) \in D$ and $J$ is a radical ideal of $D$, we have that $\varphi(a) \in J$. Thus, $\varphi \in \IntR(E, J)$. This shows that $\sqrt{\IntR(E,J)} \subseteq \IntR(E,J)$ and therefore $\IntR(E,J)$ is a radical ideal of $\IntR(E,D)$. The same argument shows that $\IntR(E,J)$ is a radical ideal of $\IntR(E,T)$. 
		
		\item 	Claim: A prime ideal of $\IntR(E,T)$ containing $\IntR(E,J)$ is a maximal ideal of $\IntR(E,T)$. 
		
		Take $\P$ to be a prime ideal of $\IntR(E,T)$ containing $\IntR(E,J)$. Then let $\varphi \in \IntR(E,T) \setminus \mathfrak{P}$. Set $\psi = \frac{\varphi^n}{t+\varphi^{2n}}$. We want to show that $\psi \in \IntR(E,T)$. Take an element $a \in E$. Then for each $\m \in \Lambda$, we have 
		\[
		v_\m(\psi(a)) = \begin{cases}
			0, & \text{if $v_\m(\varphi(a)) = 0$},\\
			nv_\m(\varphi(a)) - v_\m(t) > 0, & \text{if $v_\m(\varphi(a)) > 0$}.
		\end{cases}
		\]
		Since $T = \bigcap\limits_{\m \in \Lambda} T_\m$, we have that $\psi(a) \in T$. This holds for all $a \in E$, so $\psi \in \IntR(E,T)$.
		
		Now we observe that
		\[
		\varphi^n(1-\varphi^n\psi) = \varphi^n\cdot \frac{t+\varphi^{2n} - \varphi^n \varphi^n}{t+\varphi^{2n}} = t\psi.
		\]
		We have that $J \subseteq \mathfrak{P} \cap T$, so $\mathfrak{P} \cap T$ is a maximal ideal of $T$. Thus, $t \in \mathfrak{P} \cap T \subseteq \mathfrak{P}$, and therefore $t\psi \in \mathfrak{P}$. 
		
		We now have that $\varphi^n(1-\varphi^n\psi) \in \mathfrak{P}$. Since $\varphi \notin \mathfrak{P}$, we must have $1 - \varphi^n\psi \in \mathfrak{P}$. This means that $\varphi$ has an inverse modulo $\mathfrak{P}$. The previous statement holds for any $\varphi \in \IntR(E,T) \setminus \mathfrak{P}$, implying that $\mathfrak{P}$ is a maximal ideal in $\IntR(E,T)$. 
		
		\item 	Claim: A prime ideal of $\IntR(E,D)$ containing $\IntR(E,J)$ is a maximal ideal of $\IntR(E,D)$. 
		
		Take $\P$ to be a prime ideal of $\IntR(E,D)$ containing $\IntR(E,J)$. Then let $\varphi \in \IntR(E,D) \setminus \mathfrak{P}$. Set $\psi = \frac{\varphi^n}{t+\varphi^{2n}}$. We want to show that $\psi \in \IntR(E,D)$.
		
		Take an element $a \in E$ and a maximal ideal $\mathfrak{a}$ of $D$. Let $\mathfrak{a}'$ be the unique maximal ideal of $T$ that contracts to $\mathfrak{a}$. We know that $\mathfrak{a}'$ is an ultrafilter limit of ideals in $\Lambda$. For $\m \in \Lambda$, we have that $\psi(a) \in \m$ if and only if $\varphi(a) \in \m$, as calculated in the third claim. This implies that $\psi(a) \in \mathfrak{a}'$ if and only if $\varphi(a) \in \mathfrak{a}'$. 
		
		If $\varphi(a) \in \mathfrak{a}'$, then $\psi(a) \in \mathfrak{a}' \subseteq \mathfrak{a}' T_{\mathfrak{a}'} = \mathfrak{a}D_{\mathfrak{a}} \subseteq D_\mathfrak{a}$. 
		
		If $\varphi(a) \notin \mathfrak{a}'$, then we have $\varphi(a)^{2n} \in D \setminus \mathfrak{a}' \subseteq D_\mathfrak{a}^\times$ and combining with the fact that $t \in \mathfrak{a}' \subseteq \mathfrak{a}' T_{\mathfrak{a}'} = \mathfrak{a} D_\mathfrak{a}$ yields $t + \varphi(a)^{2n} \in D_\mathfrak{n}^\times$. Using the fact that $\varphi(a)^n \in D_\mathfrak{a}$ if $\varphi(a) \notin \mathfrak{a}'$, we get that $\psi(a) \in D_\mathfrak{a}$. 
		
		No matter if $\varphi(a) \in \mathfrak{a}'$ or $\varphi(a) \notin \mathfrak{a}'$, we get that $\varphi(a) \in D_{\mathfrak{a}}$. This holds for all $a \in E$ and all maximal ideals $\mathfrak{a}$ of $D$. Thus, $\psi \in \IntR(E,D)$. 
		
		Now we observe that
		\[
		\varphi^n(1-\varphi^n\psi) = \varphi^n\cdot \frac{t+\varphi^{2n} - \varphi^n \varphi^n}{t+\varphi^{2n}} = t\psi.
		\]
		We claim that $t\psi \in \mathfrak{P}$. Let $\mathfrak{a}$ be a maximal ideal of $D$. Let $\mathfrak{a}'$ be the unique maximal ideal of $T$ that contracts to $\mathfrak{a}$. We know $t \in \mathfrak{a} D_\mathfrak{a}$ and thus $t \in D_\mathfrak{a}$ for all maximal ideals $\mathfrak{a}$ of $D$. Thus, $t \in D$. Since $\P$ contains $\IntR(E,J)$, we have that $J \subseteq \mathfrak{P} \cap D$, so $\q_1 := \mathfrak{P} \cap D$ is a maximal ideal of $D$. Let $\q_2$ the unique maximal ideal of $T$ that contracts to $\q_1$. We know that $t \in \q_2 \subseteq \q_2 T_{\q_2} = \q_1 D_{\q_1}$. Now we see that $t \in \q_1 D_{\q_1} \cap D = \q_1 \subseteq \mathfrak{P}$. Thus, $t\psi \in \mathfrak{P}$.
		
		We now have that $\varphi^n(1-\varphi^n\psi) \in \mathfrak{P}$. Since $\varphi \notin \mathfrak{P}$, we must have $1 - \varphi^n\psi \in \mathfrak{P}$. This means that $\varphi$ has an inverse modulo $\mathfrak{P}$ for all $\varphi \in \IntR(E,D) \setminus \P$, implying that $\mathfrak{P}$ is a maximal ideal in $\IntR(E,D)$. 
		
		\item Claim: 	Every maximal ideal of $\IntR(E,T)$ is of the form $\lim\limits_{\mathcal{U}} \M_{\m,a}^T$ for some ultrafilter $\mathcal{U}$ of $\{\M_{\m, a}^T \mid a \in E, \m \in \Lambda\}$.
		
		Since $T$ is a singular Prüfer domain, we know that $\IntR(E,T)$ is a Prüfer domain \cite[Theorem 3.5]{IntValuedRational}. Then we know every maximal ideal of $\IntR(E,T)$ is a $t$-ideal. Additionally, we write 
		\[
		\IntR(E,T) = \bigcap_{a \in E} \bigcap_{\m \in \Lambda} \IntR(E,T)_{\M_{\m, a}^T}
		\]
		and then Proposition 2.8 of \cite{PvMD} proves the claim.
		
		\item  Claim:  Every maximal ideal of $\IntR(E,D)$ is of the form $\lim\limits_{\mathcal{U}} \M_{\mathfrak{n},a}^D$ for some ultrafilter $\mathcal{U}$ of $\{\M_{\mathfrak{n}, a}^D \mid a \in E, \m \in \Lambda, \mathfrak{n} = \m \cap D \}$.
		
		We will take our characteristic sets with respect to \[\{\M_{\mathfrak{n}, a}^D \mid a \in E, \m \in \Lambda, \mathfrak{n} = \m \cap D\}.\] Let $A \subseteq \IntR(E,D)$ be a proper ideal. We want to show that $U := \{\chi_\varphi \mid \varphi \in A \}$ is closed under finite intersections.
		
		Take $\varphi_1, \varphi_2 \in A$. Set \[\theta(x) = \frac{t(1+x^{2n})}{(1+tx^n)(t+x^n)}.\] We claim that $\theta \in \IntR(K,D)$.
		
		Take $a \in K$ and let $\mathfrak{p}$ be a maximal ideal of $D$. Let $\mathfrak{p}'$ be the unique maximal ideal of $T$ that contracts to $\mathfrak{p}$. We know that $\p$ is an ultrafilter limit of ideals in $\Lambda$ with respect to some ultrafilter $\mathcal{U}$ of $\Lambda$. Let $\m \in \Lambda$. We have that $v_\m(a) > 0$ implies $v_\m\left(\frac{a^n}{t}\right) > 0$, meaning $a \in \p'$ implies $\frac{a^n}{t} \in \p'$. Since $a \in \p'$ implies that $\{\m \in \Lambda \mid a \in \m \} \in \mathcal{U}$, we know that $\{\m \in \Lambda \mid  \frac{a^n}{t}\in \m \} \in \mathcal{U}$ as well. This means that $a \in \p'$ implies $\frac{a^n}{t} \in \p'$. Note that $t \in \p'$ as well. 
		
		Now we calculate. Denote by $v$ a valuation associated to $T_{\p'}$. If $v(a) = 0$, then $v(\theta(a)) = v(t) + v(1+a^{2n}) - 0 - 0 > 0$, so $\theta(a) \in \p' T_{\p'} = \mathfrak{p} D_\mathfrak{p} \subseteq D_\mathfrak{p}$. If $v(a) > 0$, then 
		\[
		\theta(a) = \frac{t(1+a^{2n})}{(1+ta^n)(t+a^n)} = \frac{1+a^{2n}}{(1+ta^n)(1+ \frac{a^n}{t})}. 
		\] 
		We have $1+a^{2n} \equiv 1 \pmod {\p'}$ and $(1+ta^n)(1+ \frac{a^n}{t}) \equiv 1 \cdot 1 \equiv 1 \pmod {\p'}$. This means that $\theta(a) \in 1+ \p' T_{\p'} \subseteq D_\mathfrak{p}$. Lastly, suppose that $v(a) < 0$. We calculate that
		\[
		\theta(a) = \frac{t(1+a^{2n})}{(1+ta^n)(t+a^n)} = \frac{\frac{1}{a^{2n}}+1}{(\frac{1}{ta^n}+1)(\frac{t}{a^n}+ 1)}. 
		\]
		We see that $\frac{1}{a^{2n}}+1 \equiv 1 \pmod {\p'}$ and also $(\frac{1}{ta^n}+1)(\frac{t}{a^n}+ 1) \equiv 1 \cdot 1 \equiv 1 \pmod {\p'}$. Thus, $\theta(a) \in 1+ \p' T_{\p'} \subseteq D_\mathfrak{p}$.  We now know that $\theta(a) \in D_\mathfrak{p}$ for all $a \in K$ and maximal ideals $\mathfrak{p}$ of $D$, so $\theta \in \IntR(K,D)$.

		Now we consider $\rho(x) = \varphi_1(x) + \theta\left( \frac{\varphi_1(x)}{\varphi_2(x)}\right) \varphi_2(x)$. Since $\theta\left( \frac{\varphi_1(x)}{\varphi_2(x)}\right) \in \IntR(K,D)$, which is contained in $\IntR(E,D)$, we have that $\rho(x) \in (\varphi_1, \varphi_2) \subseteq A$.

		Let $a \in E$. Now let $\mathfrak{m} \in \Lambda$. Suppose that $v_\m(\varphi_1(a)) = v_\m(\varphi_2(a))$. Then $v_\m(\rho(a)) = v_\m(\varphi_1(a))$. If $v_\m(\varphi_1(a)) < v_\m(\varphi_2(a)))$, we have $v_\m(\rho(a)) = v_\m(\varphi_1(a))$. If $v_\m(\varphi_1(a)) > v_\m(\varphi_2(a))$, we have $v_\m(\rho(a)) = v_\m(\varphi_2(a))$. In summary, $v_\m(\rho(a)) = \min\{v_\m(\varphi_1(a)), v_\m(\varphi_2(a)) \}$. This implies that $\chi_{\varphi_1} \cap \chi_{\varphi_2} = \chi_{\rho}$. Thus, $U$ is closed under finite intersections.
		
		Since $A$ is a proper ideal, $U$ does not contain the empty set. We have just shown that $U$ is closed under finite intersections, so we can deduce that $U$ has the finite intersection property. We can then extend $U$ to $\mathcal{U}$, an ultrafilter of $\{\M_{\mathfrak{n}, a}^D \mid a \in E, \m \in \Lambda, \mathfrak{n} =\m \cap D\}$. Then we see that $A \subseteq \lim\limits_{\mathcal{U}}\M_{\mathfrak{n}, a}^D$. Thus, all maximal ideals of $\IntR(E,D)$ are of the form $\lim\limits_{\mathcal{U}}\M_{\mathfrak{n}, a}^D$ for some ultrafilter $\mathcal{U}$ of $\{\M_{\mathfrak{n}, a}^D \mid a \in E, \m \in \Lambda, \mathfrak{n} =\m \cap D\}$.
		
		\item Claim: Suppose that $\mathcal{U}_1$ and $\mathcal{U}_2$ are ultrafilters of $\{\M_{\m, a}^T \mid a \in E, \m \in \Lambda\}$ such that $\lim\limits_{\mathcal{U}_1} \M_{\m,a}^T$ and $\lim\limits_{\mathcal{U}_2} \M_{\m,a}^T$ are distinct maximal ideals of $\IntR(E,T)$. Then $\lim\limits_{\mathcal{U}_1} \M_{\m,a}^T \cap \IntR(E, D) \neq \lim\limits_{\mathcal{U}_2} \M_{\m,a}^T \cap \IntR(E,D)$.
		
		Suppose that $\mathcal{U}_1$ and $\mathcal{U}_2$ are ultrafilters of $\{\M_{\m, a}^T \mid a \in E, \m \in \Lambda\}$ such that $\lim\limits_{\mathcal{U}_1} \M_{\m,a}^T$ and $\lim\limits_{\mathcal{U}_2} \M_{\m,a}^T$ are distinct maximal ideals of $\IntR(E,T)$. Then $\lim\limits_{\mathcal{U}_1} \M_{\m,a}^T \cap \IntR(E, D) \neq \lim\limits_{\mathcal{U}_2} \M_{\m,a}^T \cap \IntR(E,D)$.
		
		Let $\varphi \in \lim\limits_{\mathcal{U}_2} \M_{\m,a}^T \setminus \lim\limits_{\mathcal{U}_1} \M_{\m,a}^T$. Consider
		\[
		\psi(x) = \frac{\varphi(x)}{\varphi(x) + \theta(\varphi(x))}, 
		\]
		where $\theta(x) = \frac{t(1+x^{2n})}{(1+tx^n)(t+x^n)}$ from Claim 6. Take a maximal ideal $\mathfrak{p}$ of $D$. We know that $\mathfrak{p}$ is the contraction of some maximal ideal $\p'$ of $T$. Let $v$ denote a valuation corresponding to $T_{\p'}$. Now take $a \in E$. If $v(\varphi(a)) > 0$, then $v(\psi(a)) = v(\varphi(a)) - v(\theta(\varphi(a))) = v(\varphi(a)) > 0$, so $\varphi(a) \in D_\mathfrak{p}$. If $v(\varphi(a)) = 0$, then $v(\theta(\varphi(a))) > 0$ so $\varphi(a) + \theta(\varphi(a)) \in \varphi(a) + \p' T_{\p'} \subseteq D_\mathfrak{p}^\times$. This shows that $\psi(a) \in D_\mathfrak{p}$. Thus, $\psi(a) \in D$ and therefore $\psi \in \IntR(E,D)$. From this, we also see that $\varphi(a) \in \p'$ if and only if $\psi(a) \in \mathfrak{p}$. Thus, $\psi$ is in $\lim\limits_{\mathcal{U}_2} \M_{\m,a}^T \cap \IntR(E,D)$ but not in $\lim\limits_{\mathcal{U}_1} \M_{\m,a}^T \cap \IntR(E,D)$.

		\item Claim: 	The ring $\IntR(E,D)$ is a GPVD with associated Prüfer domain $\IntR(E,T)$.
		
		We know that the ideal $\IntR(E, J)$ is a nonzero radical ideal common to both $\IntR(E, D)$ and $\IntR(E,T)$ from Claim 2. Additionally, any prime ideal of $\IntR(E,D)$ containing $\IntR(E, J)$ is maximal in $\IntR(E,D)$ and any prime ideal of $\IntR(E,T)$ containing $\IntR(E, J)$ is maximal in $\IntR(E,T)$ by Claims 3 and 4. Thus,
		\[
		\dim(\IntR(E,D)/\IntR(E,J)) = \dim(\IntR(E,T)/\IntR(E,J)) = 0.
		\]
		Next, we need to show that the extension \[\IntR(E,D)/\IntR(E,J) \subseteq \IntR(E,T)/\IntR(E,J)\] is unibranched. Every maximal ideal of $\IntR(E,T)$ containing $\IntR(E,J)$ is of the form $\lim\limits_{\mathcal{U}} \M_{\m,a}^T$ for some ultrafilter $\mathcal{U}$ of $\{\M_{\m, a}^T \mid a \in E, \m \in \Lambda\}$, and two distinct ideals of this form contract to two distinct ideals of $\IntR(E,D)$ containing $\IntR(E,J)$ due to Claim 7. 
		
		Now we take a maximal ideal of $\IntR(E,D)$ containing $\IntR(E,J)$. From Claim 6, we know that this maximal ideal has the form $\lim\limits_{\mathcal{U}} \M_{\mathfrak{n},a}^D$ for some ultrafilter $\mathcal{U}$ of $\{\M_{\mathfrak{n}, a}^D \mid a \in E, \m \in \Lambda, \mathfrak{n} = \m \cap D \}$. We construct the ultrafilter
		\[
		\mathcal{U}' \coloneqq \{ \{\M_{\m, a}^T \mid \M_{\mathfrak{n},a}^D \in S, \m \cap D = \mathfrak{n} \} \mid S \in \mathcal{U} \}
		\]
		of $\{\M_{\m, a}^T \mid a \in E, \m \in \Lambda\}$. We claim that $\lim\limits_{\mathcal{U}} \M_{\mathfrak{n},a}^D = \lim\limits_{\mathcal{U}'} \M_{\mathfrak{m},a}^T \cap \IntR(E,D)$. Let $\varphi \in \lim\limits_{\mathcal{U}} \M_{\mathfrak{n},a}^D$. Then $\{\M_{\mathfrak{n},a}^D \mid \varphi(a) \in \mathfrak{n}, a \in E, \m \in \Lambda, \mathfrak{n} = \m \cap D \} \in \mathcal{U}$. Since $\mathfrak{m}\cap D \subseteq \m$, we have that $\{\M_{\m, a}^T \mid \varphi(a) \in \m, a \in E, \m \in \Lambda \} \in \mathcal{U}'$. Thus, $\varphi \in \lim\limits_{\mathcal{U}'} \M_{\mathfrak{m},a}^T \cap \IntR(E,D)$. To show the reverse inclusion, we now suppose that $\varphi \in \lim\limits_{\mathcal{U}'} \M_{\mathfrak{m},a}^T \cap \IntR(E,D)$. Then $\{\M_{\m, a}^T \mid \varphi(a) \in \m, a \in E, \m \in \Lambda \} \in \mathcal{U}'$. Since $\varphi$ is in $\IntR(E,D)$ as well, we know for any $a \in E$ and $\m \in \Lambda$ that $\varphi(a) \in \m$ implies that $\varphi(a) \in \m \cap D$. Thus, $\{\M_{\mathfrak{n},a}^D \mid \varphi(a) \in \mathfrak{n}, a \in E, \m \in \Lambda, \mathfrak{n} = \m \cap D \}$ is in $\mathcal{U}$ and therefore $\varphi \in \lim\limits_{\mathcal{U}} \M_{\mathfrak{n},a}^D$. This implies the contraction map of the prime spectra of the extension $\IntR(E,D)/\IntR(E,J) \subseteq \IntR(E,T)/\IntR(E,J)$ is surjective.

		Thus, $\IntR(E,D)/\IntR(E,J) \subseteq \IntR(E,T)/\IntR(E,J)$ is a unibranched extension. This shows that $\IntR(E,D)$ is a GPVD with associated Prüfer domain $\IntR(E,T)$.
	\end{enumerate}
\end{proof}

\section{PVDs with non-singular associated valuation domains}\label{Sect:Non-singular}

\indent\indent Now that we have seen that rings of integer-valued rational functions over pseudosingular GPVDs are GPVDs, we investigate rings of integer-valued rational functions over a base ring that is a GPVD but not a pseudosingular GPVD. We restrict our focus to the base ring being a PVD. A PVD being not pseudosingular means that its associated valuation overring is not singular, which means that the maximal ideal of the associated valuation overring is not principal. Thus, we consider the case where the base ring is a PVD whose associated valuation overring does not have a principal maximal ideal. 

For a PVD, we can make use of the valuation associated with the associated valuation overring. This allows us to utilize the tool of the minimal valuation function. Here, we view a value group $\Gamma$ as being embedded in its divisible closure $\Q\Gamma \coloneqq \Gamma \otimes_\Z \Q$.

\begin{definition}\cite{Liu}
	Let $V$ be a valuation domain with value group $\Gamma$, valuation $v$, and field of fractions $K$. Take a nonzero polynomial $f \in K[x]$ and write it as $f(x) = a_nx^n + \cdots + a_1x+ a_0$ for $a_0, a_1, \dots, a_n \in K$. We define the \textbf{minimum valuation function of $f$} as $\minval_{f}: \Gamma \to \Gamma $ by \[\gamma \mapsto \min\{v(a_0), v(a_1)+\gamma, v(a_2)+2\gamma, \dots, v(a_n) + n\gamma\}\] for each $\gamma \in \Gamma$. We may also think of $\minval_f$ as a function from $\Q{\Gamma}$ to $\Q{\Gamma}$ defined as $\gamma \mapsto \min\{v(a_0), v(a_1)+\gamma, v(a_2)+2\gamma, \dots, v(a_n) + n\gamma\}$ for each $\gamma \in \Q{\Gamma}$. 
	
	For a nonzero rational function $\varphi \in K[x]$, we write $\varphi = \frac{f}{g}$ for some $f, g \in K[x]$. Then for each $\gamma \in \Gamma$, we define $\minval_\varphi(\gamma) = \minval_f(\gamma) - \minval_g(\gamma)$.
\end{definition}

The purpose of the minimum valuation function of a rational function is that it can predict the valuation of the outputs of the rational function most of the time. The minimum valuation function also has the nice property of being piecewise linear. The following lemma showcases these facts. 

\begin{lemma}\cite[Proposition 2.24, Lemma 2.26, Proposition 2.27]{Liu}\label{Lem:MinvalForm}
	Let $V$ be a valuation domain with value group $\Gamma$, valuation $v$, maximal ideal $\m$, and field of fractions $K$. For a nonzero $\varphi \in K(x)$, the function $\minval_\varphi$ has the following form evaluated at $\gamma \in \Q\Gamma$
	\[
	\minval_\varphi(\gamma) = \begin{cases}
		c_1 \gamma + \beta_1, & \gamma \leq \delta_1,\\
		c_2 \gamma + \beta_2, & \delta_1 \leq \gamma \leq \delta_2,\\
		\vdots\\
		c_{k-1} \gamma + \beta_{k-1}, & \delta_{k-2} \leq \gamma  \leq \delta_{k-1},\\
		c_k \gamma + \beta_k, & \delta_{k-1} \leq \gamma,
	\end{cases}
	\] 
	where $c_1, \dots, c_k \in \Z$; $\beta_1, \dots, \beta_k \in \Gamma$; and $\delta_1, \dots, \delta_{k-1} \in \Q{\Gamma}$ such that $\delta_1 < \cdots < \delta_{k-1}$. 
	
	Furthermore, for all but finitely many $\gamma \in \Gamma$, we have that $v(\varphi(t)) = \minval_\varphi(v(t))$ for all $t \in K$ such that $v(t) = \gamma$. If the residue field of $V$ is infinite, then for any $\varphi_1, \dots, \varphi_n \in K(x)$ and any $\gamma \in \Gamma$, there exists $a \in K$ with $v(a) = \gamma$ such that $\minval_{\varphi_i}(\gamma) = v(\varphi_i(a))$ for all $i$. 
\end{lemma}

The following result of Dobbs and Fontana shows how the common radical ideal helps describe what the localizations of a GPVD at maximal ideals look like. 

\begin{proposition}\cite[p. 156]{DobbsFontana}\label{Prop:GPVDEquivalentDefinition}
	Let $D$ be a GPVD and $T$ the Prüfer domain associated to $D$. Then by the equivalent definition of GPVD, there exists a common radical ideal $J$ such that $D/J \subseteq T/J$ is a unibranched extension of Krull dimension 0 rings. Let $\mathfrak{n}$ be a maximal ideal of $D$ and $\mathfrak{m}$ be the maximal ideal of $T$ contracting to $\mathfrak{n}$. Then
	\begin{itemize}
		\item if $J \not\subseteq \mathfrak{m}$, then $D_\mathfrak{n} = T_\mathfrak{m}$ is a valuation domain, and
		\item if $J \subseteq \mathfrak{m}$, then $D_\mathfrak{n}$ is a PVD with associated valuation domain $T_\mathfrak{m}$. 
	\end{itemize}
\end{proposition}

Even when $D$ is a PVD that is not a valuation domain, we can define the prime ideal $\M^*$ of $\IntR(D)$ when the associated valuation domain has infinite residue field or maximal ideal that is not principal. This is defined analogously to the $\M^*$ prime ideal defined for $\IntR(V)$, where $V$ is a valuation domain with infinite residue field or maximal ideal that is not principal in \cite{IntValuedRational}. We define
\[
\M^* \coloneqq \{\varphi \in \IntR(D) \mid \minval_\varphi(0) > 0 \},
\]
where $\minval$ is defined using the valuation associated with the associated valuation domain. We check this is indeed a prime ideal of $\IntR(D)$. We will eventually only need this following result in the case where the maximal ideal of the associated valuation domain is not principal. 

\begin{lemma}
	Let $D$ be a PVD. Suppose $D$ has infinite residue field or the maximal ideal of the associated valuation domain is not principal. Then $\M^*$ is a prime ideal of $\IntR(D)$. 
\end{lemma}

\begin{proof}
	We will let $V$ denote the valuation domain associated to $D$. Also let $v$ be an associated valuation, $\m$ be the maximal ideal of $V$, and $\Gamma$ the value group. 
	
	Let $\varphi, \psi \in \M^*$. We want to show $\varphi + \psi \in \M^*$. We consider the case where $D$ has infinite residue field and the case where $\m$ is not principal in $V$ separately.
	
	Suppose that $D$ has infinite residue field. Then there exists some $u \in D$ such that $v(\varphi(u)) = \minval_\varphi(0), v(\psi(u)) = \minval_\psi(0),$ and $v((\varphi+\psi)(u)) = \minval_{\varphi+\psi}(0)$ by Lemma \ref{Lem:MinvalForm}. Then we have
	\begin{align*}
		\minval_{\varphi+\psi}(0) = v((\varphi+\psi)(u)) & \geq \min\{v(\varphi(u)), v(\psi(u))\} \\&= \min\{\minval_\varphi(0), \minval_\psi(0) \} \\&> 0. 
	\end{align*}
	Therefore, $\varphi+\psi \in \M^*$. 
	
	If $\m$ is not principal in $V$, then there exists $\varepsilon \in \Gamma$ with $\varepsilon > 0$ such that for all $d \in D$ such that $0 < v(d) < \varepsilon$, we have $
	\minval_\varphi(v(d)) = v(\varphi(d)), \minval_\psi(v(d)) = v(\varphi(d)),$ and $\minval_{\varphi+\psi}(v(d)) = v((\varphi+\psi)(d))$ by Lemma \ref{Lem:MinvalForm}. This implies that
	\[
	\minval_{\varphi+\psi}(\gamma) \geq \min\{\minval_\varphi(\gamma), \minval_\psi(\gamma) \},
	\]
	for all $\gamma \in \Gamma$ such that $0 < \gamma < \varepsilon$. Thus, the above inequality also holds for $\gamma = 0$ by Lemma \ref{Lem:MinvalForm}, which leads to $\minval_{\varphi+\psi}(0) \geq \min\{\minval_\varphi(0), \minval_\psi(0) \} > 0$. Thus, $\varphi+\psi \in \M^*$. 
	
	Now let $\rho \in \IntR(D)$ and $\varphi \in \M^*$. We can use similar techniques as above to show that $\minval_\rho(0) \geq 0$. Then $\minval_{\rho\varphi}(0) = \minval_{\rho}(0) + \minval_\varphi(0) > 0$. This shows that $\rho\varphi \in \M^*$, so $\M^*$ is an ideal of $\IntR(D)$. 
	
	Now suppose that $\rho, \rho'\in \IntR(D)$ such that $\rho\rho'\in \M^*$. We see that \[\minval_{\rho\rho'}(0) = \minval_{\rho}(0) + \minval_{\rho'}(0) > 0,\] which implies that $\minval_{\rho}(0) > 0$ or $\minval_{\rho'}(0) > 0$ since $\minval_{\rho}(0) \geq 0$ and $\minval_{\rho'}(0) \geq 0$. Thus, $\M^*$ is a prime ideal of $\IntR(D)$. 
\end{proof}

Now we focus on the case where the maximal ideal of the associated valuation domain is not principal. The following lemma is then analogous to Theorem 6.6 of \cite{IntValuedRational}. 

\begin{lemma}\label{Lem:MStarNotMaximal}
	Let $D$ be a PVD with associated valuation domain $V$ whose maximal ideal is not principal. Then the prime ideal $\M^*$ is not maximal. 
\end{lemma}

\begin{proof}
	Denote by $v$ a valuation associated with $V$ and $\Gamma$ its value group. Let $\mathcal{U}$ be a non-principal ultrafilter of $\{\M_{ \m, a} \mid a \in D \}$ containing the family of sets of ideals $\{ \{ \M_{ \m, a} \mid a \in \m, v(a) < \varepsilon \}\mid \epsilon \in \Gamma, \epsilon > 0 \}$. We claim that $\M^* \subsetneq \lim\limits_{\mathcal{U}} \M_{\m, a}$. 
	
	Let $\varphi \in \M^*$. Then there exist $\varepsilon\in \Gamma$ with $\varepsilon > 0$, $c \in \Z$, and $\beta \in \Gamma$ such that $\minval_\varphi(\gamma) = c\gamma + \beta$ for all $\gamma$ with $0 < \gamma < \varepsilon$ and for all $d \in D$ such that $0 < v(d) < \varepsilon$, we have
	\[
	v(\varphi(d)) = \minval_\varphi(v(d)). 
	\]
	We can make $\varepsilon$ small enough so that $c\gamma + \beta > 0$ for all $\gamma$ such that $0 < \gamma < \varepsilon$ since $\minval_\varphi(0) = \beta > 0$. This shows that $\varphi \in \lim\limits_{\mathcal{U}} \M_{\m, a}$. 
	
	The containment $\M^* \subseteq \lim\limits_{\mathcal{U}} \M_{\m, a}$ is strict since $x \in \lim\limits_{\mathcal{U}} \M_{\m, a} \setminus \M^*$. 
\end{proof}

Now we show that for a PVD $D$ that is not a valuation domain and is not pseudosingular, then $\IntR(D)$ is not a GPVD. 

\begin{proposition}\label{Prop:NotPseudosingularPVD}
	Let $D$ be a PVD with associated valuation domain $V$ whose maximal ideal is not principal and $D \neq V$. Then the domain $\IntR(D)$ is not a GPVD.
\end{proposition}

\begin{proof}
	Let $\m$ denote the maximal ideal of $D$. Also let $a \in D$. We first show that $\IntR(D)_{\M_{\m, a}}$ is not a valuation domain. We have that 
	\[
	\IntR(D)_{\M_{\m, a}} \subseteq \{\varphi \in K(x) \mid \varphi(a) \in D \}. 
	\]
	Intersecting with $K$ yields that $\IntR(D)_{\M_{\m, a}} \cap K \subseteq D$. Take $d \in V \setminus D$. Then $d, d^{-1} \in V \setminus D$, so neither $d$ nor $d^{-1}$ is in $\IntR(D)_{\M_{\m, a}}$. Thus, $\IntR(D)_{\M_{\m, a}}$ is not a valuation domain. 
	
	Now suppose for a contradiction that $\IntR(D)$ is a GPVD with associated Prüfer domain $T$ and common radical ideal $J$ such that $\IntR(D)/J \subseteq T/J$ is a unibranched extension of Krull dimension 0 rings. By Proposition \ref{Prop:GPVDEquivalentDefinition}, we see that $J \subseteq \M_{\m, a}$ for all $a \in D$, since $\IntR(D)_{\M_{\m, a}}$ is not a valuation domain. Since $J \subseteq \M_{\m, a}$ for all $a \in D$, we have $J \subseteq \IntR(D,\m)$. Together with $\IntR(D,\m) \subseteq \M^*$, we get $J \subseteq \M^*$, but by Lemma \ref{Lem:MStarNotMaximal}, we have a prime ideal of $\IntR(D)$ containing $J$ that is not maximal, contradicting the assumption that $\IntR(D)$ is a GPVD.
	
\end{proof} 

The previous proposition required the PVD to be not a valuation domain. If the base ring $V$ is a valuation domain, we know exactly when $\IntR(V)$ is a Prüfer domain. 

\begin{theorem}\label{Thm:IntRVClassification}\cite[Corollary 2.30]{Liu}
	Let $V$ be a valuation domain with maximal ideal $\m$. Then $\IntR(V)$ is a Prüfer domain if and only if $V/\m$ is not algebraically closed or $\m$ is a principal ideal of $V$.
\end{theorem}

Since a Prüfer domain is a GPVD, we know that if a valuation domain $V$ has a principal maximal ideal or a residue field that is not algebraically closed, then $\IntR(V)$ is a GPVD. This means that a valuation domain $V$ can be not (pseudo)singular and $\IntR(V)$ is still a GPVD as long as the residue field of $V$ is not algebraically closed. 

We also know that if $V$ is a valuation domain with algebraically closed residue field and maximal ideal that is not principal, then $\IntR(V)$ is not a Prüfer domain. We show that if the value group associated with $V$ is not divisible, then $\IntR(V)$ is also not a GPVD in this case. We show this by first showing that a domain that is not a Prüfer domain but an essential domain cannot be a GPVD. 

\begin{proposition}
	Let $D$ be a domain that is not a Prüfer domain with a family of essential maximal ideals $\{\m_\lambda\}_{\lambda \in \Lambda}$ such that $\bigcap\limits_{\lambda \in \Lambda} D_{\m_\lambda}$. Then $D$ is also not a GPVD.
\end{proposition}

\begin{proof}
	Suppose on the contrary that $D$ is a GPVD. Then $D$ has an associated Prüfer overring $T$. Now for all $\lambda \in \Lambda$, let $\mathfrak{n}_\lambda$ be the unique maximal ideal of $T$ that contracts to $\m_\lambda$. Then by Proposition \ref{Prop:GPVDEquivalentDefinition}, we deduce that $D_{\m_\lambda} = T_{\mathfrak{n}_\lambda}$ because $D_{\m_\lambda}$ is a valuation domain. 
	Thus,
	\[
	T \subseteq \bigcap_{\lambda \in \Lambda} T_{\mathfrak{n}_\lambda} = \bigcap_{\lambda \in \Lambda} D_{\mathfrak{m}_\lambda} = D,
	\]
	showing that $T = D$, but we assumed that $D$ is not Prüfer, so this is a contradiction, meaning $D$ cannot be a GPVD.
\end{proof}

\begin{corollary}\label{Cor:IntRNotPruferNotGPVD}
	Let $D$ be a domain, $K$ its field of fractions, and $E$ a subset of $K$. Assume that $\IntR(E,D)$ is not Prüfer, but there is a family of maximal ideals $\{\m_\lambda \}_{\lambda \in \Lambda}$ of $D$ such that $D = \bigcap\limits_{\lambda \in \Lambda} D_{\m_\lambda}$ and each ideal in $\{\M_{\m_\lambda, a} \mid \lambda \in \Lambda, a \in E \}$ is essential. Then $\IntR(E,D)$ is not a GPVD. 
\end{corollary}

\begin{proof}
	Because $\IntR(E,D)_{\M_{\m_\lambda, a}} \subseteq \{ \varphi \in K(x) \mid \varphi(a) \in D_{\m_\lambda} \}$ for every $\lambda \in \Lambda $ and $a \in E$, we know that $\IntR(E,D) = \bigcap\limits_{\lambda \in \Lambda}\bigcap\limits_{a \in E}\IntR(E,D)_{\M_{\m_\lambda, a}}$. Therefore, $\IntR(E,D)$ is not a GPVD. 
\end{proof}

In order to use the previous corollary, we have to show that the maximal pointed ideals are essential. 

\begin{proposition}\label{Prop:PointedMaximalIdealsAreEssential}\cite[Proposition 2.35]{Liu}
	Let $V$ be a valuation domain whose value group is not divisible. Let $\m$ be the maximal ideal of $V$, $E$ be a subset of $K$, the field of fractions of $V$, and take $a \in E$. Then 
	\[\IntR(E, V)_{\mathfrak{M}_{\m, a}} = \{\varphi \in K(x) \mid \varphi(a) \in V \},\] a valuation domain. 
\end{proposition}

\begin{corollary}
	Suppose that $V$ is a valuation domain such that the maximal ideal is not principal, the residue field is algebraically closed, and the value group is not divisible. Then $\IntR(V)$ is not a GPVD. 
\end{corollary}

\begin{proof}
	We know that $\IntR(V)$ is not a Prüfer domain due to Theorem \ref{Thm:IntRVClassification}. The maximal pointed ideals of $\IntR(V)$ are essential due to Proposition \ref{Prop:PointedMaximalIdealsAreEssential}. Together, this allows us to apply Corollary \ref{Cor:IntRNotPruferNotGPVD} to show that $\IntR(V)$ is not a GPVD.
\end{proof}

Let $D$ be a PVD with associated valuation domain $V$. First suppose that $D \neq V$. Then $\IntR(D)$ is a GPVD if and only if the maximal ideal is principal. If $\IntR(D)$ is a GPVD, then the associated Prüfer domain is $\IntR(D,V)$. Interestingly, if $D \neq V$ and $V$ has a residue field that is not algebraically closed, then $\IntR(D,V)$ is a Prüfer domain \cite[Theorem 3.2]{IntValuedRational}, but $\IntR(D)$ is not a GPVD. 

Now suppose the base ring is a valuation domain $V$. If the residue field of $V$ is not algebraically closed or the maximal ideal is principal, then $\IntR(V)$ is a Prüfer domain and thus a GPVD. If the residue field of $V$ is algebraically closed, the maximal ideal of $V$ is not principal, and the value group associated with $V$ is not divisible, then $\IntR(V)$ is not a GPVD. The remaining case is the case where the residue field of $V$ is algebraically closed and the value group associated with $V$ is divisible.

\section{Local rings of integer-valued rational functions}\label{Sect:Local}

\indent\indent In this section, we give a family of rings of integer-valued rational functions over PVDs that are local domains. This uses a result about rational functions as maps between fields. First, we give a lemma that shows that if we have a valuation that in a sense separates the units and the maximal ideal of a local domain, then we get a ring of integer-valued rational functions that is local. 

\begin{lemma}\label{Lem:MinvalDichotomyImpliesLocal}\cite[Lemma 3.4]{IntRFactor}
	Let $D$ be a local domain with maximal ideal $\m$ and field of fractions $K$. Suppose that there is a valuation overring $V$ of $D$ with a valuation $v$ such that for all $d \in \m$, we have $v(d) > 0$. Also let $\Gamma$ be the value group of $v$. Suppose that for all $\varphi \in \IntR(K,D)$ that $\minval_\varphi(\gamma) = 0$ for all $\gamma \in \Gamma$ or $\minval_\varphi(\gamma) > 0$ for all $\gamma \in \Gamma$. Then $\IntR(K,D)$ is a local domain with maximal ideal $\IntR(K,\m)$.
\end{lemma}

To create this dichotomy in the minimum valuation function when the base ring is a PVD, we make use a lemma about rational functions as maps between from a larger field to a smaller field in a field extension. The field extension of interest here is the one from the residue field of the PVD to the residue field of the associated valuation domain. 

\begin{definition}
	Let $L/M$ be a purely inseparable field extension of fields of characteristic $p > 0$. We say that $L/M$ is of \textbf{finite exponent} if there exists some $e \in \N$ such that $a^{p^e} \in M$ for all $a \in L$. 
\end{definition}

The proof of the following lemma will be in Section \ref{Sect:FieldMaps}.

\begin{lemma}\label{Lem:NoRatMapToSubfield2}
	Let $L/M$ be a field extension that is not purely inseparable of finite exponent. Additionally, suppose that $L$ is an infinite field. Then there does not exist a nonconstant rational function $\varphi \in L(x)$ such that $\varphi(d) \in M$ for all but finitely many $d \in L$. 
\end{lemma}

This fact about rational functions as maps between fields will be used alongside the polynomials that fall out considering the residue fields of the PVD and the associated valuation domain. These polynomials are called local polynomials.

\begin{definition}\cite{Liu}
	Let $V$ be a valuation domain with an associated valuation $v$ and field of fractions $K$. Take $f \in K[x]$ be a nonzero polynomial and $t \in K$. We define the \textbf{local polynomial of $f$ at $t$} to be \[\loc_{f, v, t}(x) = \frac{f(tx)}{a_dt^d} \mod \m,\] where  $d = \max\{i \in \{0, 1, \dots, n\} \mid v(a_i) + iv(t) = \minval_f(v(t)) \}$ and $\m$ is the maximal ideal of $V$. This is a well-defined monic polynomial with coefficients in $V/\m$.
\end{definition}

One utility of the local polynomial is that it can determine the coefficients that appear in the minimal valuation function. 

\begin{lemma}\label{Lem:PowersOfLoc}\cite[Lemma 2.25]{Liu}
	Take $\varphi \in K(x)$ to be nonzero and $\alpha \in \Gamma$. There exist $\varepsilon \in \Q\Gamma$ with $\varepsilon > 0$ small enough, $c, c' \in \Z$, and $\beta, \beta' \in \Gamma$ such that
	\[
	\minval_\varphi(\gamma) = \begin{cases}
		c\gamma + \beta, & \text{if $\alpha - \varepsilon < \gamma < \alpha$},\\
		c'\gamma + \beta', & \text{if $\alpha < \gamma < \alpha + \varepsilon$}. 
	\end{cases}
	\]
	Write $\varphi = \frac{f}{g}$ for some $f, g \in K[x]$. Take $t \in K$ such that $v(t) = \gamma$. We can write $\loc_{f, t} = a_{i_1}x^{i_1} + \cdots + a_{i_r}x^{i_r}$ and $\loc_{g, t} = b_{j_1}x^{j_1}+ \cdots + b_{j_s}x^{j_s}$ for some nonzero $a_{i_1}, \dots, a_{i_r}, b_{j_1}, \dots, b_{j_s} \in V/\m$. Then
	\[
	c = i_r - j_s \quad \text{and} \quad c' = i_1-j_1.
	\]
\end{lemma}

We now show that the ring of integer-valued rational functions over a PVD on its field of fractions can be local under certain conditions on the residue fields. 

\begin{proposition}\label{Prop:PVDConditionsMakeIntRLocal}
	Let $D$ be a PVD with $V$ being the corresponding valuation overring, and $\m$ being the common maximal ideal. Suppose that $\Gamma$, the value group of $V$, is divisible. Let $M := D/\m$ and $L := V/\m$. Suppose further that $L/M$ is not purely inseparable of finite exponent and $L$ is infinite. Then $\IntR(K,D)$ is local with maximal ideal $\IntR(K,\m)$.
\end{proposition}

\begin{proof}
	Let $v$ be a valuation associated with $V$. Take $\varphi \in \IntR(K,D)$. Write $\varphi = \frac{f}{g}$ with $f, g \in D[x]$. Take $\gamma \in \Gamma$ to be some element and $t \in K$ such that $v(t) = \gamma$. Then $\loc_{f, t}, \loc_{g, t} \in L[x]$. We want to show that $\frac{\loc_{f, t}}{\loc_{g, t}}$ maps $L$ to $M$. Write $f = \sum\limits_i a_i x^i$ and $g = \sum\limits_j b_j x^j$ with $a_i, b_j \in D$. Pick out all the indices $i_1 < \cdots < i_r$ such that $v(a_{i_1}t^{i_1}) = \cdots = v(a_{i_r}t^{i_r}) = \minval_f(\gamma)$, and similarly pick out all the indices $j_1 < \cdots < j_s$ such that $v(b_{j_1}t^{j_1}) = \cdots = v(b_{j_s}t^{j_s}) = \minval_g(\gamma)$. We write
	\[
	\frac{b_{j_s}}{a_{i_r}}t^{j_s-i_r}\varphi(tx) = \frac{f(tx)/(a_{i_r}t^{i_r})}{g(tx)/(b_{j_s}t^{j_s})} = \frac{\frac{a_{i_1}t^{i_1}x^{i_1} + \cdots + a_{i_r}t^{i_r}x^{i_r} + \sum\limits_{i\neq i_1, \dots, i_r}a_ix^i}{a_{i_r}t^{i_r}}}{\frac{b_{j_1}t^{j_1}x^{j_1} + \cdots + b_{j_s}t^{j_s}x^{j_s} + \sum\limits_{j\neq j_1, \dots, j_s}b_ix^j}{b_{j_s}t^{j_s}}}.
	\]
	Let $c \in L$ such that $c$ is not a root of neither $\loc_{f, t}$ nor $\loc_{g, t}$. Note that all but finitely many elements of $L$ satisfy this condition. Take $u \in V$ such that $u + \m = c$. Then
	\[
	\frac{f(tu)}{a_{i_r}t^{i_r}} \mod \m = \frac{a_{i_1}t^{i_1}u^{i_1} + \cdots + a_{i_r}t^{i_r}u^{i_r} + \sum\limits_{i\neq i_1, \dots, i_r}a_iu^i}{a_{i_r}t^{i_r}} \mod \m = \loc_{f, t}(c) \neq 0.
	\]
	Similarly,
	\[
	\frac{g(tu)}{b_{j_s}t^{j_s}} \mod \m = \loc_{g, t}(c) \neq 0. 
	\]
	Therefore,
	\[
	\varphi(tu) \mod \m = \left(\frac{a_{i_r}}{b_{j_s}}t^{i_r-j_s} \mod \m \right) \frac{\loc_{f, t}}{\loc_{g, t}}(c). 
	\]
	Whenever $\minval_\varphi(\gamma) = 0$, we have that $v\left(\frac{a_{i_r}}{b_{j_s}}t^{i_r-j_s} \right) = 0$, so $\left(\frac{a_{i_r}}{b_{j_s}}t^{i_r-j_s} \mod \m \right) \neq 0$. Since $\varphi(tu) \in D$, we obtain that $\left(\frac{a_{i_r}}{b_{j_s}}t^{i_r-j_s} \mod \m \right)\frac{\loc_{f, t}}{\loc_{g, t}}(c) \in M$. This holds for all but finitely many $c \in L$, so $\left(\frac{a_{i_r}}{b_{j_s}}t^{i_r-j_s} \mod \m \right)\frac{\loc_{f, t}(x)}{\loc_{g, t}(x)}$ is constant by Lemma \ref{Lem:NoRatMapToSubfield2}. Moreover, $\frac{\loc_{f, t}(x)}{\loc_{g, t}(x)}$ is constant. We claim that this shows that the existence of some $\gamma \in \Gamma$ such that $\minval_\varphi(\gamma) = 0$ implies that $\minval_\varphi = 0$. If there exist $\delta, \delta' \in \Gamma$ such that $\minval_\varphi(\delta) = 0$ and $\minval_\varphi(\delta') > 0$, we can assume without loss of generality that $\delta < \delta'$. Then there exist $\alpha, \beta, \varepsilon \in \Gamma$ and $n \in \Z\setminus\{0\}$ such that
	\[
	\minval_\varphi(\gamma) = \begin{cases}
		n\gamma + \beta, & \alpha \leq \gamma \leq \alpha + \varepsilon \\
		0, & \alpha-\varepsilon \leq \gamma \leq \alpha
	\end{cases}
	\]
	by Lemma \ref{Lem:MinvalForm}. Note that $\alpha \in \Gamma$ because $\Gamma$ is divisible. Now take $t \in K$ such that $v(t) = \alpha$. The fact that $\frac{\loc_{f, t}(x)}{\loc_{g, t}(x)}$ is constant implies that $n-0 = 0$, by Lemma \ref{Lem:PowersOfLoc}. This is a contradiction. Thus, Lemma \ref{Lem:MinvalDichotomyImpliesLocal} implies that $\IntR(K,D)$ is local with maximal ideal $\IntR(K,\m)$.

\end{proof}

Note that $\IntR(K,D)$ is not trivial in this case. As an example, for any rational function $\varphi \in \IntR(K,V)$ and $d \in \m$, we have $d\varphi \in \IntR(K,\m) \subseteq \IntR(K,D)$. 

The following proposition shows that without the conditions on the extension of residue fields of Proposition \ref{Prop:PVDConditionsMakeIntRLocal}, the ring $\IntR(K,D)$ can be not local.

\begin{proposition}\label{Prop:PVDConditionsMakeIntRNotLocal}
	Let $D$ be a PVD with associated valuation domain $V \neq D$ and shared maximal ideal $\m$. Set $M: = D/\m$ and $L:= V/\m$. Suppose that $L$ is finite or $L/M$ is purely inseparable of finite exponent. Then $\IntR(K,D)$ is not local.
\end{proposition}

\begin{proof}
	If $L$ is finite of order $q$, then $x^q - x + 1$ maps $L$ to $\{1\} \subseteq M$. We claim then that $\frac{1}{x^q-x+1} \in \IntR(K,D)$. If $a \in K$ with $v(a) < 0$, then $v\left(\frac{1}{a^q-a+1}\right) = -qv(a) > 0$. If $a \in K$ with $v(a) \geq 0$, then $\frac{1}{a^q-a+1} \in 1 + \m$, so $\frac{1}{a^q-a+1} \in D$. We see then that $\frac{1}{x^q-x+1} \in \mathfrak{M}_{\m, a} \subseteq \IntR(K,D)$ for all $a \in K$ with $v(a) < 0$ and $\frac{1}{x^q-x+1} \notin \mathfrak{M}_{\m, a} \subseteq \IntR(K,D)$ for all $a \in K$ with $v(a) \geq 0$. This means $\IntR(K,D)$ is not local.
	
	Now consider the case when $L/M$ is purely inseparable of finite exponent. This means that there exists $e \in \N$ such that $a^{p^e} \in M$ for all $a \in L$, where $p > 0$ is the characteristic of $M$. Since $V \neq D$, we have that $L \neq M$. Let $c \in L\setminus M$. Then the polynomial $x^{p^e} - c$ has no roots in $L$. Additionally, the polynomial $(x^{p^e} - c)^{p^e} = x^{p^{2e}} - c^{p^e}$ also has no roots in $L$ and has coefficients in $M$. Consider the rational function $\frac{1}{x^{p^{2e}} - u^{p^e}}$, where $u +\m = c$. Let $a \in K$ such that $v(a) < 0$. Then $v\left(\frac{1}{a^{p^{2e}}-u^{p^e}} \right) = -p^{2e}v(a)>0$. If $a \in K$ is such that $v(a) \geq 0$, then since $x^{p^{2e}} - c^{p^e}$ has no roots in $L$, we calculate that $v\left(\frac{1}{a^{p^{2e}}-u^{p^e}} \right) = 0$. Furthermore, $a^{p^{2e}}-u^{p^e} + \m \in M$, so $\frac{1}{a^{p^{2e}}-u^{p^e}} \in D$. This shows that $\frac{1}{x^{p^{2e}} - u^{p^e}} \in \IntR(K,D)$. We have $\frac{1}{x^{p^{2e}} - u^{p^e}} \in \M_{ \m, a} \subseteq \IntR(K,D)$ for all $a \in K$ with $v(a) < 0$ and $\frac{1}{x^{p^{2e}} - u^{p^e}} \notin \M_{ \m, a} \subseteq \IntR(K,D)$ for all $a \in K$ with $v(a) \geq 0$. Thus, $\IntR(K,D)$ is not local in this case. 
\end{proof}

In conclusion, for a PVD $D$ that is not a valuation domain with field of fractions $K$, we can determine exactly when $\IntR(K,D)$ is a local domain. 

\begin{corollary}
	Let $D$ be a PVD with associated valuation domain $V$, maximal ideal $\m$, and field of fraction $K$. Suppose that $D \neq V$ and set $M \coloneqq D/\m$ and $L \coloneqq V/\m$. Then $\IntR(K,D)$ is a local domain if and only if $L/M$ is not purely inseparable of finite exponent and $L$ is infinite. 
\end{corollary}

\section{Rational functions as maps between fields}\label{Sect:FieldMaps}
\indent\indent Let $L/M$ be an extension of fields. This section leads up to a lemma showing
the nonexistence of a nonconstant rational function that maps L to M, except if L
is finite or $L/M$ is a particular type of purely inseparable extension. This lemma is used for Proposition \ref{Prop:PVDConditionsMakeIntRLocal}. 

First we discuss polynomials with coefficients in the base field.

\begin{lemma}\label{Lem:NoPolyMapToSubfield}
	Let $L/M$ be a field extension that is not purely inseparable of finite exponent. Additionally, suppose that $L$ is an infinite field. Then there does not exist a nonconstant polynomial $f \in M[x]$ such that $f(d) \in M$ for all but finitely many $d \in L$. 
\end{lemma}

\begin{proof}
	Suppose first that $M = \{d_1, \dots, d_k\}$ is finite. Also assume the existence of $f \in M[x]$ such that $f$ maps all but finitely many elements of $L$ to $M$. Then $(f(x)-d_1)\cdots(f(x)-d_k)$ evaluates to 0 for all but finitely many elements of $L$, which is infinite, so $(f(x)-d_1)\cdots(f(x)-d_k) = 0$. This implies that $f$ is a constant. 
	
	Now we suppose that $M$ is infinite. Suppose for a contradiction that there is a nonconstant polynomial $f \in M[x]$ such that $f(x) \in M$ for all but finitely many $x \in L$. Write $f(x) = a_0 + a_1x + a_2x^2 + \cdots + a_mx^m$ where each $a_i \in M$ with $a_m \neq 0$. Additionally, $m \geq 2$ since if $m = 1$, then $L = M$, but then $L/M$ would be purely inseparable of finite exponent.  
	
	Let $\alpha \in L \setminus M$ such that $f(\alpha) \in M$. Then $\alpha$ is algebraic over $M$. Suppose $\alpha$ has degree $n \geq 2$ over $M$. For $i \in \N$, we can uniquely represent $\alpha^i = \sum\limits_{j=0}^{n-1} c_{ij} \alpha^j$, where each $c_{ij}$ is in $M$. Then 
	\[
	f(\alpha x) = \sum_{i=0}^{m} a_i (\alpha x)^i = \sum_{i=0}^m a_i \sum_{j=0}^{n-1} c_{ij} \alpha^j x^i  = \sum_{j=0}^{n-1}\left(\sum\limits_{i=0}^m a_i c_{ij} x^i\right) \alpha^j.
	\]
	Evaluating $f(\alpha x)$ at $x = d$ for all but finitely many $d \in M$ gives an $M$-linear combination of $1, \alpha, \alpha^2, \dots, \alpha^{n-1}$, which by assumption is in $M$, so the coefficients of $\alpha, \alpha^2, \dots, \alpha^{n-1}$ must be 0. Since $\sum\limits_{i=0}^m a_i c_{ij} x^i$ evaluates to 0 at all but finitely many elements of the infinite field $M$ for $j=1, \dots, n-1$, we deduce that $\sum\limits_{i=0}^m a_i c_{ij} x^i = 0$ for $j=1, \dots, n-1$. This implies that $a_ic_{ij} = 0$ for $i = 0, \dots, m$ and $j=1, \dots, n-1$. Since $a_m \neq 0$, we have $c_{mj} = 0$ for all $j=1, \dots, n-1$. This means that $\alpha^m \in M$. Thus, the polynomial $x^m$ maps all but finitely many elements of $L$ to $M$. 
	
	Take $f \in M[x]$ to be a nonconstant polynomial such that $f$ maps all but finitely many elements of $L$ to $M$ and $f$ is a polynomial with minimal degree with respect to this property. From above, we know that $x^m$ maps all but finitely many elements of $L$ to $M$, where $m = \deg(f)$. We know that $(x+1)^m$ maps all but finitely many elements of $L$ to $M$ as well. This means that $(x+1)^m - x^m = \sum\limits_{i=0}^{m-1} \binom{m}{i} x^i$ maps all but finitely many elements of $L$ to $M$. Since $m$ was chosen to be minimal, we have that $\deg ((x+1)^m - x^m) = 0$ and thus $\binom{m}{i} = 0$ for all $i = 1, \dots, m-1$.
	
	Let $p$ be the characteristic of $M$. If $p = 0$, then $\binom{m}{i} = 0$ cannot happen. Thus, suppose that $p > 0$. Since $\binom{m}{1} = m = 0$, we have that $p$ divides $m$. Suppose that $m \neq p^r$ for any power $r \in \Z_{>0}$. Then when we write the base $p$ expansion $m = m_kp^k + m_{k-1}p^{k-1} + \cdots + m_1p +m_0$ with $0 \leq m_i < p-1$ and $m_k \neq 0$. Also, $m_0 = 0$ since $p$ divides $m$. Since $m \neq p^k$, we get that $m_k \geq 2$ or $m_i \neq 0$ for some $i = 1, \dots, k-1$. Either way, $m_i \neq 0$ for some $i = 1, \dots, k$ and $0 < p^i < m$. Then Lucas's Theorem says $\binom{m}{p^i} \equiv \binom{m_k}{0} \cdots \binom{m_{i+1}}{0}\binom{m_i}{1}\binom{m_{i-1}}{0} \cdots \binom{m_0}{0} \equiv m_i \not\equiv 0 \pmod p$. Thus, $\binom{m}{p^i} \neq 0 \in M$, a contradiction. Thus $m = p^r$ for some $r \in \Z_{>0}$. 
	
	However, this implies that $x^{p^r}$ maps all but finitely elements of $L$ to $M$. Let $\alpha \in L$. There exists some nonzero $c \in M$ such that $(c\alpha)^{p^r}$ because $M$ is infinite. Then $c^{p^r}\alpha^{p^r} \in M$ and thus $\alpha^{p^r} \in M$. Therefore, $x^{p^r}$ maps all elements of $L$ to $M$, meaning that $L/M$ is purely inseparable. If $L/M$ is not of finite exponent, then there exists $d \in L$ such that $e[d:M] > r$, contradicting the fact that $x^{p^r}$ maps $L$ to $M$. Thus, a nonconstant polynomial $f \in M[x]$ cannot map $L$ into $M$. 
\end{proof}

\begin{lemma}\label{Lem:NoRatMapToSubfield}
	Let $L/M$ be a field extension that is not purely inseparable of finite exponent. Additionally, suppose that $L$ is an infinite field. Then there does not exist a nonconstant rational function $\varphi \in M(x)$ such that $\varphi(d) \in M$ for all but finitely many $d \in L$. 
\end{lemma}

\begin{proof}
	First, we will handle the case when $M = \{d_1, \dots, d_k\}$ is finite. If $\varphi \in M(x)$ such that $\varphi$ maps all but finitely many elements of $L$ to $M$, then $(\varphi-d_1) \cdots (\varphi-d_k)$ evaluates to 0 for all but finitely element of $L$, which is infinite. Therefore, $(\varphi-d_1) \cdots (\varphi-d_k) = 0$, forcing $\varphi$ to be constant.
	
	Now we assume that $M$ is infinite for here on. Suppose there exists a nonconstant $\varphi \in M(x)$ such that $\varphi(d) \in M$ for all but finitely many $d \in L$. Write $\varphi = \frac{f}{g}$ with $f, g \in M[x]$ coprime. Let $d \in L$. Since $M$ is infinite, there exists an $a \in M$ such that $\varphi(ad) \in M$. Since $ad$ is a root of $f(x) - \varphi(ad)g(x) \in M[x]$, we know that $ad$ is algebraic over $M$ and thus $d$ is algebraic over $M$, or $f(x) - \varphi(ad)g(x) = 0$. However, if $f(x) - \varphi(ad)g(x) = 0$, then $\varphi(x) = \varphi(ad)$, which is impossible since $\varphi$ is assumed to be nonconstant. This shows that $L/M$ is an algebraic field extension. 
	
	Let $\alpha \in L \setminus M$ and let $n$ be the degree of $\alpha$ over $M$. Set $\psi(x_0, x_1, \dots, x_{n-1}) \coloneqq \varphi(x_0 + x_1\alpha + x_2\alpha^2 + \cdots + x_{n-1}\alpha^{n-1}) = \frac{f(x_0 + x_1\alpha + x_2\alpha^2 + \cdots + x_{n-1}\alpha^{n-1})}{g(x_0 + x_1\alpha + x_2\alpha^2 + \cdots + x_{n-1}\alpha^{n-1})}$, where $x_0, \dots, x_{n-1}$ are indeterminates, and write
	\[
	f(x_0 + x_1\alpha + x_2\alpha^2 + \cdots + x_{n-1}\alpha^{n-1}) = f_0 + f_1\alpha +f_2\alpha^2 + \cdots + f_{n-1}\alpha^{n-1}
	\]
	and
	\[
	g(x_0 + x_1\alpha + x_2\alpha^2 + \cdots + x_{n-1}\alpha^{n-1}) = g_0 + g_1\alpha +g_2\alpha^2 + \cdots + g_{n-1}\alpha^{n-1},
	\]
	where $f_i, g_i \in M[x_0, x_1, \dots, x_{n-1}]$ since $1, \alpha, \alpha^2, \dots, \alpha^{n-1}$ are linearly independent over $M$, so $1, \alpha, \alpha^2, \dots, \alpha^{n-1}$ are linearly independent over $M(x_0, x_1, \dots, x_{n-1})$.
	
	Now we evaluate $\psi$ at $\mathbf{a} = (a_0, \dots, a_{n-1})$, where $a_0, \dots, a_{n-1} \in M$. We get that
	\[
	\frac{f_0(\mathbf{a}) + f_1(\mathbf{a})\alpha + \cdots + f_{n-1}(\mathbf{a})\alpha^{n-1}}{g_0(\mathbf{a}) + g_1(\mathbf{a})\alpha + \cdots + g_{n-1}(\mathbf{a})\alpha^{n-1}} = \psi(\mathbf{a}) =  \varphi(a_0 + a_1\alpha + \cdots + a_{n-1}\alpha^{n-1}) \in M. 
	\]
	Because $\psi(\mathbf{a})$ is in $M$ and $1, \alpha, \dots, \alpha^{n-1}$ are linear independent over $M$, for each $i = 0, 1, \dots, n-1$, we have $f_i(\mathbf{a}) = \psi(\mathbf{a})g_i(\mathbf{a})$. Then $f_i - \psi g_i$ is a rational function that evaluates to 0 for all but finitely many $\mathbf{a} \in M^n$, which means that $f_i - \psi g_i = 0$. Rearranging yields $f_i = \psi g_i$ for $i = 0, 1, \dots, n-1$. 
	
	Consider $g(x_0 + x_1\alpha + x_2\alpha^2 + \cdots + x_{n-1}\alpha^{n-1}) = g_0 + g_1\alpha +g_2\alpha^2 + \cdots + g_{n-1}\alpha^{n-1}$. We then have $g(x) = g_0(x, 0, \dots, 0) + g_1(x,0,\dots, 0)\alpha + \cdots + g_{n-1}(x,0,\dots,0)\alpha^{n-1} \in M[x]$. Thus, $g(x) = g_0(x,0,\dots, 0)$. In particular, $g_0$ is not identically zero. Using $f_0 = \psi g_0$ from before, we conclude that $\frac{f_0}{g_0} = \psi$. This implies that
	\[
	\frac{f_0 + f_1\alpha +f_2\alpha^2 + \cdots + f_{n-1}\alpha^{n-1}}{g_0 + g_1\alpha +g_2\alpha^2 + \cdots + g_{n-1}\alpha^{n-1}}  = \frac{f_0}{g_0}. 
	\]
	Cross-multiplying and subtracting yields
	\[
	(f_1g_0 - f_0g_1)\alpha + \cdots + (f_{n-1}g_0 - f_0 g_{n-1})\alpha^{n-1}= 0
	\]
	Suppose for a contradiction that there exists $i \in \{1, 2, \dots, n-1\}$ such that $g_i \neq 0$. Then $f_i g_0 - f_0 g_i = 0$ implies that $\frac{f_0}{g_0} = \frac{f_i}{g_i}$. Thus, for all but finitely many $d \in M$, we have $\varphi(d) = \frac{f_0}{g_0}(d, 0, \dots, 0) = \frac{f_i}{g_i}(d, 0, \dots, 0) = 0$. The last equality is due to the fact that $f(x) \in M[x]$ implies $f_j(x, 0, \dots, 0) = 0$ for all $j \in \{1,2,\dots, n-1\}$. This forces $\varphi = 0$, but we assumed that $\varphi$ is nonconstant, so this is a contradiction. Therefore, $g_1 = \cdots = g_{n-1} = 0$. This means $g(x_0 + x_1\alpha + x_2\alpha^2 + \cdots + x_{n-1}\alpha^{n-1}) = g_0(x_0, \dots, x_{n-1})$ and thus $g(d)$ is a polynomial with coefficients in $M$ such that $g(d) \in M$ for all but finitely many $d \in M[\alpha]$.
	
	If $L/M$ is not purely inseparable, we can choose $\alpha \in L$ to be separable so by Lemma \ref{Lem:NoPolyMapToSubfield}, the polynomial $g$ is constant. If $L/M$ is purely inseparable of infinite exponent, then $g$ must be constant as well, since otherwise, $\deg g > 0$ implies that $\deg g \geq [M(\alpha): M] = p^{e[\alpha:M]}$ for each $\alpha \in L$, which is unbounded. In both cases, $\varphi$ is a polynomial in $M[x]$ such that $\varphi(d) \subseteq M$ for all but finitely many $d \in L$, but applying Lemma \ref{Lem:NoPolyMapToSubfield} again shows that $\varphi$ must be constant, a contradiction.
\end{proof}

Our goal now is to upgrade the previous lemma to the nonexistence of such a rational function in $L(x)$. 

\begin{proposition}\label{Prop:FieldExtensionUnnecessary}\cite[Proposition X.1.4]{Cahen}
	Let $D$ be a domain with field of fractions $K$, $E$ be an infinite subset of $K$, and $L$ be a field extension of $K$. If $\varphi \in L(x)$ is such that $\varphi(E) \subseteq D$, then, in fact, $\varphi \in K(x)$. 
\end{proposition}

The following is a stronger version of Lemma \ref{Lem:NoRatMapToSubfield}. When $L/M$ is a field extension that is not purely inseparable with $L$ being infinite, not only does there not exist a nonconstant rational function $\varphi \in M(x)$ such that $\varphi(d) \in M$ for all but finitely many $d \in L$, but there does not exist such a rational function in $L(x)$ either. 

\begin{lemma}
	Let $L/M$ be a field extension that is not purely inseparable of finite exponent. Additionally, suppose that $L$ is an infinite field. Then there does not exist a nonconstant rational function $\varphi \in L(x)$ such that $\varphi(d) \in M$ for all but finitely many $d \in L$. 
\end{lemma}

\begin{proof}
	Proceeding with proof by contradiction, let $\varphi \in L(x)$ be a nonconstant rational function such that $\varphi(x) \in M$ for all but finitely many $x \in L$. 
	
	First, we will handle the case when $M = \{d_1, \dots, d_k\}$ is finite. If $\varphi \in L(x)$ such that $\varphi(L) \subseteq M$, then $(\varphi-d_1) \cdots (\varphi-d_k)$ evaluates to 0 for all but finitely element of $L$, which is infinite. Therefore, $(\varphi-d_1) \cdots (\varphi-d_k) = 0$, forcing $\varphi$ to be constant.
	
	Now we assume that $M$ is infinite. Following the notation in Proposition \ref{Prop:FieldExtensionUnnecessary}, we let $D = M$ and $L = L$. Also let $E$ be the set of elements $d \in M$ such that $\varphi(d) \in M$. Note that $E$ is infinite. Then $\varphi(E) \subseteq M$, so $\varphi \in M(x)$. However, $\varphi \in M(x)$ is a rational function such that $\varphi(d) \in M$ for all but finitely many $d \in L$, a contradiction of Lemma \ref{Lem:NoRatMapToSubfield}.
\end{proof}

\bibliographystyle{amsalpha}
\bibliography{references}
\end{document}